\documentclass{amsart}
\usepackage{amssymb,amsbsy, amsthm, amsmath,mathtools, amsfonts}
\usepackage{color}
\usepackage{bookmark}
\usepackage{tikz}
\usepackage{pgfplots}
\usepackage{amsrefs}

\usepackage{subcaption}

\pgfplotsset{width=7cm,compat=1.9}

\def\R{\mathbb{R}}
\def\N{\mathbb{N}}
\def\C{\mathbb{C}}

\theoremstyle{plain}
\newtheorem{theorem}{Theorem}[section]

\newtheorem{corollary}[theorem]{Corollary}

\newtheorem{proposition}[theorem]{Proposition}

\theoremstyle{definition}
\newtheorem{definition}[theorem]{Definition}
\newtheorem*{remark}{Remark}

\usepackage{hyperref}
\usepackage{physics}
\usepackage{amsrefs}

\begin{document}

\title[Chebyshev functions, multipoint Pad\'e approximants, and noise]{A remark on Chebyshev rational functions, multipoint Pad\'e approximants and Noise}

\author[M.~Derevyagin]{Maxim~Derevyagin}
\address{
MD,
Department of Mathematics\\
University of Connecticut\\
341 Mansfield Road, U-1009\\
Storrs, CT 06269-1009, USA}
\email{maksym.derevyagin@uconn.edu}

\author[M. Meynig]{Max Meynig}
\address{MM, Department of Physics, University of Connecticut, 196A Auditorium Road Unit 3046, Storrs, 06269, CT, United States}
\email{max.meynig@uconn.com}


\begin{abstract} 
Motivated by the recent interest in multipoint Padé approximants in the physics community, we discuss Chebyshev rational functions and show how they give rise to multipoint Padé approximants in exactly the same way that Chebyshev polynomials produce Padé approximants. We present recurrence relations for Chebyshev rational functions, as well as the underlying continued fraction of Thiele type (also known as $R_{II}$ type). Finally, we provide numerical evidence illustrating the effects of noise on this interpolation scheme and show that a phenomenon similar to that recently observed by Costin, Dunne, and Meynig for Padé approximants also occurs in the multipoint setting.
\end{abstract}
\maketitle

\section{Introduction}

Pad\'e approximants have numerous applications in mathematics and physics often due to their ability to analytically continue power series data beyond the radius of convergence.
Similarly, multipoint Pad\'e approximants, a type of rational interpolant, offer similar ability to analytically continue functions only known at a finite number of points and as a result have been investigated as a tool for calculations in theoretical physics.
At the same time there have been recent advances in computational mathematics, see \cites{Trefethen_AAA_2018,Trefethen_AAA_2024,Nakatsukasa_2026} which develop the AAA algorithm, an interpolation-based rational approximation with connections to multipoint Pad\'e approximation.
Returning to physics, multipoint Pad\'e approximants have been explored as a possible tool in Lattice gauge theory such as for the computation of the anomalous magnetic moment of the muon \cite{PadeGminusTwo}.
Calculations performed in Lattice gauge theory produce function values at a finite number of `imaginary times,' and analytic continuation to `real times' is a major challenge \cite{PhysRevD.108.074516}.
In this context multipoint Pad\'e approximants are particularly appealing due to physical constraints which imply the approximated functions are Nevanlinna \cite{PhysRevD.108.074516} or Stieltjes \cite{PadeGminusTwo} and thereby are convergent {with well characterized error (See \cite{PhysRevD.108.074516}*{Section V.} for a discussion of the Pick criteria and the `Wertevorrat' in the context of Lattice QCD calculations).}
As pointed out in \cite{PadeGminusTwo} the fact that the interpolated functions are only known to finite precision, with possible systematic errors for example those due to finite lattice spacing as discussed in \cite{PhysRevD.108.074516}*{Section IV.}, complicate the convergence.



Numerical error, which can be modeled as the presence of random noise in the input data, can dramatically affect the behavior of Pad\'e approximants.
How exactly noise effects Pad\'e approximants is an interesting mathematical question and has received attention. 
For example the work of Gilewicz and Pindor \cites{GILEWICZ1997199,GILEWICZ1999285} characterized the behavior of Pad\'e approximants to rational functions in the presence of random noise.
Additionally, the work of Bessis \cite{BESSIS199685} which explored the application of multipoint Pad\'e approximants to the fitting of experimental data. 
Specifically, in \cite{BESSIS199685} Bessis applied the Thiele algorithm to data coming from electron scattering experiments and proposed the use of Pad\'e for filtering noise.
The use of Pad\'e approximants in noise filtering has further been explored by Bessis and Perotti in \cite{Bessis_2009}.
 
More recently, the work of Costin, Dunne and Meynig in \cite{Costin:2022hgc} provided a characterization of the \textit{noise induced breakdown} of Pad\'e approximants, where the addition of a small random noise causes `spurious' poles and zeros, or Froissart doublets, to accumulate on the unit circle once a critical order $[n_c-1/n_c]$ is surpassed.
A fundamental result of \cite{Costin:2022hgc} is a prediction for this critical order.
The results of \cite{Costin:2022hgc} utilize the convergence in capacity of Pad\'e approximants which provides a precise connection between a Pad\'e approximant of a function and conformal maps.
This connection allows for properties of near diagonal Pad\'e approximants in the presence of noise to be understood through properties of the conformal map from the extremal domain to the unit circle.
Large order/small noise asymptotics reduce the critical order to a simple function completely determined by the conformal map.

Less is known about the behavior of multipoint Pad\'e approximants in the presence of noise.
However, similar convergence-in-capacity results are known for multipoint Pad\'e approximants \cites{StahlHerbert1989Otco,BARATCHART2009187}, which suggests that the results of \cite{Costin:2022hgc} extend straightforwardly.
To provide evidence for this numerical experiments of the behavior of multipoint Pad\'e approximants when a small random noise is added to the interpolation data are presented in section \ref{sec::EffectOfNoise}.

Additionally, while a substantial amount of results are known about the behavior and convergence of multipoint Pad\'e approximants (see \cites{Baker,BARATCHART2009187,D11, GoncharLopez1978,StahlHerbert1989Otco}, among others) there are relatively few examples (e.g. \cites{BW,Zhedanov2004,Zhedanov2004_Elliptic}) where multipoint Pad\'e approximants can be written down explicitly such as is possible in the series case.
Here we show that the Chebyshev rational functions introduced in \cite{Borwein1994} provide such an example and lead to an explicit construction of multipoint Pad\'e approximant.
Additionally, we provide several new results on these functions such as a novel three term recurrence relation as well as proofs of uniform convergence for some Newtonian and triangular sets of interpolation points.

\section{Chebyshev polynomials}
\label{sec:ChebyshevPolynomials}

The \emph{Chebyshev polynomials of the first kind} $\{T_n\}_{n=0}^{\infty}$ and
the \emph{second kind} $\{U_n\}_{n=0}^{\infty}$ are defined on $[-1,1]$ via the trigonometric relations
\begin{equation}\label{T_def}
  T_n(\cos\theta)=\cos(n\theta), \qquad n=0,1,2,\dots,
\end{equation}
\begin{equation}\label{U_def}
  U_n(\cos\theta) =\frac{\sin\!\big((n+1)\theta\big)}{\sin\theta},
  \qquad n=0,1,2,\dots
\end{equation}
One can easily check that
\[
  T_0(x)=1,\quad T_1(x)=x
\]
and 
\[  
  U_0(x)=1,\quad U_1(x)=2x.
\]
In addition, one can derive the three-term recurrence relation for $T_n$ and $U_n$ using  standard trigonometric identities. Namely, the cosine addition identity gives
\[
  \cos\big((n+1)\theta\big)=2\cos\theta\,\cos(n\theta)-\cos\big((n-1)\theta\big).
\]
Thus, setting $x=\cos\theta$ we get 
\begin{equation}\label{Tthreterm}
  T_{n+1}(x)=2x\,T_n(x)-T_{n-1}(x), \qquad n\ge1,
\end{equation}
with initial values $T_0(x)=1$ and $T_1(x)=x$, which also reflects the fact that $T_n$'s are indeed polynomials.

Similarly, writing the sine addition identity:
\[
  \sin\big((n+2)\theta\big)=2\cos\theta\,\sin\big((n+1)\theta\big)-\sin(n\theta).
\]
and dividing both sides by $\sin\theta$ yield 
\[
  \frac{\sin\big((n+2)\theta\big)}{\sin\theta}
  = 2\cos\theta\,\frac{\sin\big((n+1)\theta\big)}{\sin\theta}
    - \frac{\sin(n\theta)}{\sin\theta}.
\]
Again, setting $x=\cos\theta$ the latter can be rewritten as 
\begin{equation}\label{Qthreeterm}
  U_{n+1}(x)=2x\,U_n(x)-U_{n-1}(x), \qquad n\ge1,
\end{equation}
with initial values $U_0(x)=1$ and $U_1(x)=2x$.

To give a different perspective, let $z$ lie on the unit circle, and set the Joukowski transform
\[
x=\frac{1}{2}\!\left(z+{z}^{-1}\right) \in [-1,1], 
\qquad z=e^{it},\ \ x=\cos t.
\]
Then, for $x\in[-1,1]$ we have
\[
T_n(x)=\frac{z^n+z^{-n}}{2},
\qquad
U_n(x)=\frac{z^{n+1}-z^{-n-1}}{z-z^{-1}},
\]
which is simply another way to write formulas \eqref{T_def} and \eqref{U_def}.
Next, for $x\in\mathbb C\setminus[-1,1]$, let $z=z(x)$ be the unique solution of
\[
x=\frac12\!\left(z+z^{-1}\right)
\]
satisfying $|z|>1$ and so the function $z=z(x)$ is the inverse Joukowski map.  
We then define, for all $x\in\mathbb C$,
\[
T_n(x)=\frac{z(x)^n+z(x)^{-n}}{2},
\qquad
U_n(x)=\frac{z(x)^{n+1}-z(x)^{-n-1}}{z(x)-z(x)^{-1}},
\]
which generalize the previous formulas beyond the $x$ interval $[-1,1]$.
Note, the map $x\mapsto z(x)$ defined by
\[
x=\frac12\!\left(z+z^{-1}\right), \qquad |z|>1,
\]
is analytic on $\mathbb C\setminus[-1,1]$. With this choice, we may define
\[
\sqrt{x^2-1}=\frac12\!\left(z(x)-z(x)^{-1}\right),
\]
which gives the branch of $\sqrt{x^2-1}$ analytic on $\mathbb C\setminus[-1,1]$
satisfying $\sqrt{x^2-1}\sim x$ as $x\to\infty$.
In particular,
\[
z(x)=x+\sqrt{x^2-1},
\qquad
z(x)^{-1}=x-\sqrt{x^2-1}.
\]
The above formulas for $T_n$ and $U_n$ become
\[
\begin{split} 
T_n(x)=\frac{(x+\sqrt{x^2-1})^n+(x-\sqrt{x^2-1})^n}{2},
\\
U_n(x)=\frac{(x+\sqrt{x^2-1})^{n+1}-(x-\sqrt{x^2-1})^{n+1}}{2\sqrt{x^2-1}},
\end{split}
\]
valid for all $x\in\mathbb C\setminus[-1,1]$.

The last property we mention here is that the ratio $\dfrac{U_{n-1}(x)}{T_n(x)}$ is 
the $[n\!-\!1/n]$ Pad\'e approximants of 
\begin{equation}\label{eq::phi}
    \phi(x)=\frac{1}{\sqrt{x^2-1}}    
\end{equation}
at infinity, where the branch of the square root is fixed by the inverse Joukowski map.
To explain what it means observe that
\[
\phi(x)=\frac{2}{z-z^{-1}}
= z^{-1}+z^{-3}+z^{-5}+\cdots,
\quad z\to\infty
\]
and at the same time
\[
\frac{U_{n-1}(x)}{T_n(x)}
=\frac{2}{z-z^{-1}}\,
\frac{1-z^{-2n}}{1+z^{-2n}}
= z^{-1}+z^{-3}+\cdots+z^{-(2n-1)}
+O(z^{-(2n+1)}).
\]
Thus the Laurent expansions at infinity agree through order $2n-1$ and so we also have
\begin{equation}\label{ChebToPade}
\phi(x)-\frac{U_{n-1}(x)}{T_n(x)}=O(x^{-(2n+1)}),
\end{equation}
which identifies $\dfrac{U_{n-1}(x)}{T_n(x)}$  as the $[n\!-\!1/n]$ Pad\'e approximant of $f$.
Moreover, since $|z(x)|>1$ uniformly on compact subsets of
$\mathbb C\setminus[-1,1]$, the error $\left|\phi(x)-\frac{U_{n-1}(x)}{T_n(x)}\right|$ decays geometrically, giving locally uniform convergence. Another way to look at this approximation property is through continued fractions. Namely, it is known that numerators and denominators of the convergents to a continued fraction satisfy a three-term recurrence relation subject to two distinct initial conditions (for instance, see \cite{Baker}*{Section 4.1}). Therefore, \eqref{Tthreterm} and \eqref{Qthreeterm} allow us to reverse engineer the underlying continued fraction. In particular, we have that 
\begin{equation}\label{defCF}
\frac{U_{n-1}(x)}{T_n(x)}=\frac{1}
{x-
 \displaystyle{\frac{1} {2x-
\displaystyle{\frac{1}{\ddots-
\displaystyle{\frac{1}{2x}}}}}}}.
\end{equation}
As a result, \eqref{ChebToPade} means that the underlying continued fraction converges and 
\[
\phi(x)=\frac{1}
{x-
 \displaystyle{\frac{1} {2x-
\displaystyle{\frac{1}{2x-
\displaystyle{\frac{1}{2x-
\ddots}}}}}}}.
\]

\section{ The Chebyshev rational functions}

In this section, we will recast the construction of Chebyshev rational functions and present recurrence relations for these Chebyshev rational functions. The Chebyshev rational functions have prescribed poles and, hence, the essence of the construction is actually the numerators, which are polynomials. These polynomials appeared in some works of S.N. Bernstein and later studied by N.I. Akhiezer, see \cite{Akhiezer}*{Section A.5}. The form of rational functions was given by the construction in \cite{Borwein1994}, where they they showed that many properties of Chebyshev polynomials remained valid for the rational case.

 To begin with, let us fix a sequence of {\it distinct} real numbers $\{a_k\}_{k=1}^{\infty}$ such that
\[
a_k\notin[-1,1].
\]
Then define another sequence of real numbers via
\[ 
c_k = a_k - \sqrt{a_k^2-1}.
\]
Clearly, $|c_k|<1$. Equivalently, $a_k$ and $c_k$ are related by the Joukowski transformation
\[
a_k =\frac{1}{2}\left(c_k+\frac{1}{c_k}\right).
\]
Let $z$ lie on the unit circle, and set
\[
x=\frac{1}{2}\!\left(z+\frac{1}{z}\right) \in [-1,1], 
\qquad z=e^{it},\ \ x=\cos t.
\]
Define the finite Blaschke product
\[
f_n(z)=\prod_{k=1}^{n} \frac{z-c_k}{1-{c_k}\,z},\qquad f_0(z)\equiv 1.
\]
Now we can introduce the generalization of the Chebyshev polynomials in the following way: 
\begin{equation}\label{TQ_def}
f_n(e^{it})=T_n(\cos t)+i\,Q_n(\cos t)
\end{equation}
Notice that if all numbers $a_k=\infty$ the above-defined $T_n$ coincides with the Chebyshev polynomial $T_n$ defined earlier and $Q_n(\cos t)=U_n(\cos t)\sin t$. 
Next, we see that if $z=e^{it}$ one has $|f_n(z)|=1$ and therefore we have
\[
f_n(e^{it})=e^{\,i\theta_n(t)},\qquad 
\theta_n(t)=\sum_{k=1}^{n}\delta_k(t),
\]
where $\delta_k(t)$ is given by
\begin{equation}\label{deltak_def}
e^{\,i\delta_k(t)}=\frac{e^{it}-c_k}{1-{c_k}\,e^{it}}.
\end{equation}
Consequently, we arrive at
\begin{equation}\label{TQ_sum}
T_n(\cos t)=\cos\theta_n(t),\qquad 
Q_n(\cos t)=\sin\theta_n(t).
\end{equation}
Next, using \eqref{TQ_sum} and the standard trigonometric relations, we obtain
\begin{equation}\label{TQ_system}
\begin{aligned}
T_{n+1}(\cos t)&=T_n(\cos t)\,\cos\delta_{n+1}(t)\;-\;Q_n(\cos t)\,\sin\delta_{n+1}(t),\\
Q_{n+1}(\cos t)&=T_n(\cos t)\,\sin\delta_{n+1}(t)\;+\;Q_n(\cos t)\,\cos\delta_{n+1}(t).
\end{aligned}
\end{equation}
It is a bit more complicated than in the case of the Chebyshev polynomials but it still allows us to land on a three term recurrence relation in this generalized case.
\begin{theorem} The functions $T_n(x)$'s are rational and they satisfy the three-term recurrence relation:
\begin{equation}\label{3termForT}
\begin{split}
\frac{x-a_{n+1}}{\sqrt{a_{n+1}^2-1}}T_{n+1}(x)+
\Big[\left(\frac{a_n}{\sqrt{a_n^2-1}}+\frac{a_{n+1}}{\sqrt{a_{n+1}^2-1}}\right)x\\
-\frac{1}{\sqrt{a_n^2-1}}-\frac{1}{\sqrt{a_{n+1}^2-1}}
\Big]T_n(x)+\frac{x-a_{n}}{\sqrt{a_{n}^2-1}}T_{n-1}(x)=0.
\end{split}
\end{equation}
\end{theorem}
\begin{proof}
First, multiplying the first equation in \eqref{TQ_system} by $\cos\delta_{n+1}(t)$ and the second one by $\sin\delta_{n+1}(t)$, and adding the resulting relations give
\[
\cos\delta_{n+1}(t)T_{n+1}(\cos t)+\sin\delta_{n+1}(t)Q_{n+1}(\cos t)=T_n(\cos t).
\] 
The latter yields the following
\begin{equation*}
\begin{aligned}
\frac{\cos\delta_{n+1}(t)}{\sin\delta_{n+1}(t)}T_{n+1}(\cos t)+Q_{n+1}(\cos t)&=\frac{1}{\sin\delta_{n+1}(t)}T_n(\cos t),\\
\frac{\cos\delta_{n}(t)\cos\delta_{n+1}(t)}{\sin\delta_{n}(t)}T_{n}(\cos t)+\cos\delta_{n+1}(t)Q_{n}(\cos t)&=\frac{\cos\delta_{n+1}(t)}{\sin\delta_{n}(t)}T_{n-1}(\cos t).
\end{aligned}
\end{equation*}
Next, subtracting the second relation from the first one and taking into account the second one from \eqref{TQ_system}, we arrive at
\[
\begin{split}\frac{\cos\delta_{n+1}(t)}{\sin\delta_{n+1}(t)}T_{n+1}(\cos t)+
\left(\sin\delta_{n+1}(t)-\frac{\cos\delta_{n}(t)\cos\delta_{n+1}(t)}{\sin\delta_{n}(t)}-\frac{1}{\sin\delta_{n+1}(t)}\right)T_n(\cos t)\\
+\frac{\cos\delta_{n+1}(t)}{\sin\delta_{n}(t)}T_{n-1}(\cos t)=0.
\end{split}
\]
Since
\[
\sin\delta_{n+1}(t)-\frac{1}{\sin\delta_{n+1}(t)}=-\frac{\cos^2\delta_{n+1}(t)}{\sin\delta_{n+1}(t)},
\]
the above recurrence relation reduces to
\[
\frac{1}{\sin\delta_{n+1}(t)}T_{n+1}(\cos t)
-\left(\frac{\cos\delta_{n}(t)}{\sin\delta_{n}(t)}+\frac{\cos\delta_{n+1}(t)}{\sin\delta_{n+1}(t)}\right)T_n(\cos t)
+\frac{1}{\sin\delta_{n}(t)}T_{n-1}(\cos t)=0.
\]
Now, from \eqref{deltak_def} we see that
\begin{equation}\label{TUzero}
\cos\delta_{k}(t)+i\sin\delta_{k}(t)=\frac{(1+c_k^2)\cos t-2c_k}{1-2c_k\cos t+c_k^2}+i\frac{(1-c_k^2)\sin t}{1-2c_k\cos t+c_k^2}.
\end{equation}
As a result, we have 
\[
\begin{split}
\frac{1+c_{n+1}^2-2c_{n+1}x}{1-c_{n+1}^2}T_{n+1}(x)-
\left[\left(\frac{1+c_n^2}{1-c_n^2}+\frac{1+c_{n+1}^2}{1-c_{n+1}^2}\right)x-
\left(\frac{2c_n}{1-c_n^2}+\frac{2c_{n+1}}{1-c_{n+1}^2}\right)
\right]T_n(x)\\+\frac{1+c_{n}^2-2c_{n}x}{1-c_{n}^2}T_{n-1}(x)=0,
\end{split}
\]
where $x=\cos t$. Notice that in addition to
\[
\frac{1}{2}\left(c_k+\frac{1}{c_k}\right)=a_k
\]
we also have that
\[
\frac{1}{2}\left(c_k-\frac{1}{c_k}\right)=- \sqrt{a_k^2-1},
\]
which allows us to rewrite the three term recurrence relation in terms of $a_k$ 
as in the statement of theorem. 
%
\end{proof}
\begin{remark}
The recurrence coefficients for Chebyshev polynomials play a role of the reference coefficients in the spectral theory of Jacobi matrices. A spectral theory of linear pencils of Jacobi matrices related to interpolation problems was developed in \cites{DZh,D11,D17}. The coefficients in \eqref{3termForT} could play a similar role of reference coefficients. In particular, they give an insight as to how the coefficients behave and what to expect in the possible formulation of an analog of the Denisov-Rakhmanov theorem in the linear pencil case. Note that an analog of the Denisov-Rakhmanov theorem was proved in \cite{D17}, but it was in the case corresponding to the multiple interpolation at $\pm i$. 
\end{remark}

Mimicking the polynomial case, let us introduce the Chebyshev functions of the second kind as follows
\[
U_n(\cos t)=\frac{\sin\theta_{n+1}(t)}{\sin t}=\frac{Q_{n+1}(\cos t)}{\sin t}.
\]
Similarly to what we did for $T_n$'s, we can establish the following property.

\begin{theorem} The functions $U_n$'s are rational and they satisfy the three-term recurrence relation:
\begin{equation}\label{3termForU}
\begin{split}
\frac{x-a_{n+1}}{\sqrt{a_{n+1}^2-1}}U_{n}(x)+
\left[\left(\frac{a_n}{\sqrt{a_n^2-1}}+\frac{a_{n+1}}{\sqrt{a_{n+1}^2-1}}\right)x-
\frac{1}{\sqrt{a_n^2-1}}-\frac{1}{\sqrt{a_{n+1}^2-1}}
\right]U_{n-1}(x)\\+\frac{x-a_{n}}{\sqrt{a_{n}^2-1}}U_{n-2}(x)=0.
\end{split}
\end{equation}
\end{theorem}
\begin{proof} Using the same steps but expressing $Q_n$ instead of $T_n$ we get
\[
\frac{1}{\sin\delta_{n+1}(t)}Q_{n+1}(\cos t)
-\left(\frac{\cos\delta_{n}(t)}{\sin\delta_{n}(t)}+\frac{\cos\delta_{n+1}(t)}{\sin\delta_{n+1}(t)}\right)Q_n(\cos t)
+\frac{1}{\sin\delta_{n}(t)}Q_{n-1}(\cos t)=0.
\]
\end{proof}

\begin{remark}
Since $f_0=1$, it is evident that
\[
T_0(x)=1, \quad U_0(x)=0.
\]
Then recall that $T_1(\cos t)+i Q_1(\cos t)=e^{i\delta_1(t)}$ and thus using \eqref{TUzero} yields
\[
T_1(x)=\frac{1-a_1x}{x-a_1}, \quad U_1(x)=- \frac{\sqrt{a_1^2-1}}{x-a_1}.
\]
These initial conditions along with the recurrence relations show that $T_n$ and $U_n$ are rational functions.
\end{remark}

The above remark and the basic theory of continued fractions lead to the following.
\begin{corollary} We have that
\begin{equation}\label{ratCF}
\frac{U_{n-1}(x)}{T_n(x)}=\frac{1}
{\frac{a_1x-1}{\sqrt{a_1^2-1}}-
 \displaystyle{\frac{(x-a_1)^2/(a_1^2-1)} {\frac{a_1x-1}{\sqrt{a_1^2-1}}+\frac{a_2x-1}{\sqrt{a_2^2-1}}-
\displaystyle{\frac{(x-a_2)^2/(a_2^2-1)}{\ddots-
\displaystyle{\frac{(x-a_{n-1})^2/(a_{n-1}^2-1)}{\frac{a_{n-1}x-1}{\sqrt{a_{n-1}^2-1}}+\frac{a_nx-1}{\sqrt{a_n^2-1}}}}}}}}}.
\end{equation}
\end{corollary}

The continued fraction \eqref{ratCF} degenerates to \eqref{defCF} when all nodes $a_k\to\infty$. Also, the continued fraction \eqref{ratCF} looks like an even or odd part of a Thiele interpolating fraction. These type of continued fraction was studied by M. Ismail and D. Masson \cite{IsmailMasson} in relation to orthogonality and they called them continued fractions of type $R_{II}$.

\section{The case of complex poles}
\label{sec::ComplexPoles}

In this section, we extend the previous construction to the case of complex poles. 
It is worth noting that although our extension is similar to what was done in \cite{Borwein1994}, it is slightly different. This means that some formulas remain valid while others require a modification or an adjustment.

Now we assume that $a_k$'s are distinct and   
\[
a_{k}\in \C\setminus [-1,1], \quad k=1, 2,\dots.
\] 
Then we can define the Chebyshev functions $T_n$ and $U_n$ as follows
\begin{equation}\label{ChebRat_c}
    \begin{aligned}
        T_n(x) &= \frac{1}{2}\left(f_n(z)+ \frac{1}{f_n(z)} \right),
     \\ U_n(x) &= \frac{1}{z-z^{-1}} \left(f_{n+1}(z)-\frac{1}{f_{n+1}(z)}  \right).
    \end{aligned}
\end{equation}
The latter reduces to the previous definition when $a_{k}\in\R$. Moreover, we have the following

\begin{theorem}\label{thm::RecurrenceRelation}
    The functions \eqref{ChebRat_c} satisfy the following recurrence relation:
     \begin{equation}
        \frac{x-a_{n+1}}{\sqrt{a_{n+1}^2-1}} u_{n+1}(x) + \left( \frac{ a_{n+1}x -1 }{\sqrt{a_{n+1}^2-1}} +\frac{ a_{n}x -1 }{\sqrt{a_{n}^2-1}}  \right) u_n(x) +\frac{x-a_{n}}{\sqrt{a_{n}^2-1}} u_{n-1}(x)=0
   \end{equation}
 Moreover, a solution $u_n$ to the above recurrence relation becomes $T_n$ if we impose the initial conditions
\[
u_0=1,\quad u_1=\frac{1-a_1x}{x-a_1}
\]
and $u_n$ subject to the initial conditions
\[
u_{0}=0,\quad u_1=-\frac{\sqrt{a_1^2-1}}{x-a_1}
\]
coincides with $U_n$.
\end{theorem}

\begin{proof}
From the real case we know that $T_n$'s satisfy a recurrence relation and so we expect this to hold in the case of complex nodes. Therefore we just need to check that the equation $T_{n+1} + q(x)  T_{n} +p(x) T_{n-1} = 0$ has a solution in $p$ and $q$.
  
 Next, the observation 
    \begin{equation}
        f_n(z) = \frac{z-c_{n} }{1-c_n z} f_{n-1}(z)
    \end{equation}
    implies that
    \begin{equation}
        T_{n+1}(x) =\frac{1}{2}\left[ \frac{z-c_{n+1} }{1-c_{n+1} z} \frac{z-c_{n} }{1-c_{n} z} f_{n-1} + \frac{1-c_{n+1} z}{z-c_{n+1} } \frac{1-c_{n} z}{z-c_{n} } f_{n-1}^{-1} \right].
    \end{equation}
    A similar calculation for $T_n(x)$ gives the following 
     \begin{equation}
        T_{n}(x) =\frac{1}{2}\left[ \frac{z-c_{n} }{1-c_{n} z} f_{n-1} +\frac{1-c_{n} z}{z-c_{n} } f_{n-1}^{-1} \right].
    \end{equation}
       This implies the equation $T_{n+1} + q(x)  T_{n} +p(x) T_{n-1} = 0$ is equivalent to the following linear system
    \begin{equation}
        \begin{aligned}
            0 &= \frac{z-c_{n+1} }{1-c_{n+1} z} \frac{z-c_{n} }{1-c_{n} z} + q(x) \frac{z-c_{n} }{1-c_{n} z} + p(x),
        \\  0 &= \frac{1-c_{n+1} z}{z-c_{n+1} } \frac{1-c_{n} z}{z-c_{n} } + q(x)  \frac{1-c_{n} z}{z-c_{n} } + p(x).
        \end{aligned}
    \end{equation}
    The above equations have solution 
    \begin{equation}
        \begin{aligned}
            q(x) &= \frac{\left(c_n c_{n+1}-1\right) \left(c_{n+1} \left(\left(z^2+1\right) c_n-2 z\right)-2 z c_n+z^2+1\right)}{\left(c_n^2-1\right) \left(z-c_{n+1}\right) \left(z c_{n+1}-1\right)},
         \\ p(x) &=  \frac{\left(c_{n+1}^2-1\right) \left(z-c_n\right) \left(z c_n-1\right)}{\left(c_n^2-1\right) \left(z-c_{n+1}\right) \left(z c_{n+1}-1\right)} .
        \end{aligned}
    \end{equation}
    It is now straightforward using the definitions of $c_{k}$ and $z$ to show that the equation $T_{n+1} + q(x)  T_{n} +p(x) T_{n-1} = 0$ is equivalent to the recurrence relation in the statement of the theorem.
    The proof is similar for $U_n(x)$.
\end{proof}

\begin{corollary}
    For $n \in \N$ the functions $T_n$ and $U_n$ are rational functions in $x$.
\end{corollary}

The following result generalizes theorem 1.2 part (e) of \cite{Borwein1994} to the case of complex $a_k$.

\begin{proposition}
    With $a_n \in \C\setminus [-1,1]$ and $U_n$ and $T_n$ as in \ref{ChebRat_c} the following is true:
    \begin{equation}\label{PellAbel}
     [T_n(x)]^2 + (1-x^2)[U_{n-1}(x)]^2 = 1.   
    \end{equation}
\end{proposition}

\begin{proof}
    Noting that $z-z^{-1} = - 2 \sqrt{x^2 - 1}$ implies 
    \begin{equation*}
    \begin{aligned}
        & [T_n(x)]^2 + (1-x^2)[U_{n-1}(x)]^2 = \frac{1}{4} \left(f_n(z) + \frac{1}{f_n(z)}\right)^2 -\frac{1}{4}\left(f_n(z) - \frac{1}{f_n(z)}\right)^2
        \\ & = \frac{1}{4} \left( f_n(z)^2 + 2 +  \frac{1}{f_n(z)^2} -  \left( f_n(z)^2 - 2 +  \frac{1}{f_n(z)^2} \right) \right)
        \\ & = 1.
    \end{aligned}
    \end{equation*}
\end{proof}

Following the idea from \cite{Borwein1994} of using limits to find the corresponding residues one can express $T_n(x)$ in the following form
\begin{equation}\label{eq::TnPartialFraction}
    T_n(x) = A_{0,n} + \frac{A_{1,n}}{x-a_{1}} + \dots + \frac{A_{n,n}}{x-a_{n}} 
\end{equation}
where
\begin{equation}\label{eq::TnPartialFractionCoeffs}
    \begin{aligned}
        A_{0,n} &= \frac{(-1)^n}{2}(c_1^{-1}\dots c_n^{-1} + c_1\dots c_n),
    \\  A_{k,n} &= -\left( \frac{c_k -c_{k}^{-1}}{2} \right)^2 \prod_{\substack{j=1 \\ j \neq k } }^{n} \frac{1-c_k c_j}{c_k -c_j} .
    \end{aligned}
\end{equation}
Similarly, the Chebyshev rational function $U_n(x)$ has partial fraction expansion
\begin{equation}
    U_{n}(x) = \frac{B_{1,n+1}}{x-a_{1}} + \dots + \frac{B_{n+1,n+1}}{x-a_{n+1}} 
\end{equation}
where 
\begin{equation}\label{eq::BknDefinition}
    \begin{aligned}
        B_{k,n} &= - \frac{2}{c_k -c_{k}^{-1}}  A_{k,n} = \frac{A_{k,n}}{\sqrt{a_k^2-1}}
    \end{aligned}
\end{equation}

We can also establish the orthogonality relation, which was actually done in \cite{Akhiezer}*{Section A.5} and \cite{Borwein1994}*{Corollary 4.6} under slightly different assumptions but it can be straightforwardly applied to our complex settings. 

\begin{proposition}[\cite{Akhiezer},\cite{Borwein1994}]\label{thm::OrthogonalityRelationOne}
   We have that 
    \begin{equation}
        \int_{-1}^{1} T_n(x) \frac{1}{x-a_k} \frac{\dd x}{\sqrt{1-x^2 }} = 0.
    \end{equation}
    for $k\leq n$.
\end{proposition}

Additionally, the following relation, which is known for the Chebyshev polynomials, can be generalized to the case of rational functions. 

\begin{theorem}\label{prop::OrthogonalityRelationTwo}
    The Chebyshev rational functions $U_n$ can be found via $T_n$ in the following way
    \begin{equation}
        U_{n-1}(w) = \int_{-1}^{1}\frac{T_n(w) - T_n(t)}{t-w} \frac{\dd t }{\pi\sqrt{1-t^2}}.
    \end{equation}
\end{theorem}
\begin{proof}
    From \eqref{eq::TnPartialFraction} this becomes 
    \begin{equation}
        \begin{aligned}
            &\int_{-1}^{1}\frac{T_n(w) - T_n(t)}{t-w} \frac{\dd t }{\sqrt{1-t^2}} 
            = \sum_{k=1}^{n} \int_{-1}^1 \left(\frac{A_{k,n}}{t-a_k} - \frac{A_{k,n}}{w-a_k}\right)\frac{1}{t-w}  \frac{\dd t }{\pi\sqrt{1-t^2}}
        \\  & = \sum_{k=1}^{n} A_{k,n} \int_{-1}^1 \frac{t-w}{(w-a_k)(t-a_k)} \frac{1}{t-w}  \frac{\dd t }{\pi\sqrt{1-t^2}}
        \\  & =  \sum_{k=1}^{n} \frac{A_{k,n}}{w-a_k} \int_{-1}^1 \frac{1}{t-a_k}  \frac{\dd t }{\pi\sqrt{1-t^2}} 
         =  \frac{1}{\sqrt{a_k^2-1}}\sum_{k=1}^{n} \frac{A_{k,n}}{w-a_k} 
       \\ &  =  \sum_{k=1}^{n} \frac{B_{k,n}}{w-a_k} = U_{n-1}(w).
        \end{aligned}
    \end{equation}
    the last equality following from \eqref{eq::BknDefinition}.
\end{proof}


\section{Multipoint Pad\'e approximants}
\label{sec::ProofOfPade}

In this section we will show that the property \eqref{ChebToPade} can be extended to the case of the Chebyshev rational functions. This extension is an interpolation property and so we should start by recalling the corresponding concept. 
\begin{definition}[\cite{Baker}]\label{def::GeneralizedPade}
    The multipoint Pad\'e approximant of type $[L/M]$ to the function $\phi$ at the points $\{a_1 ,\dots,a_{L+M+1}\}$ is the ratio $\frac{Q_{[L/M]}}{P_{[L/M]}}$ of polynomials $Q_{[L/M]}$ and $P_{[L/M]}$ of degrees $L$ and $M$ respectively such that $P_{[L/M]}(x) \phi(x) - Q_{[L/M]}(x)$ vanishes at the points $a_1, \dots a_{L+M+1}$.
\end{definition}


\begin{remark}
    In the event of confluence, that is if any two $a_i$ are equal, the Pad\'e approximant interpolates the function value $\phi(a_i)$ and the derivative $\phi'(a_i)$ at that point. In the case of complete confluence $a_1 = a_2 = \dots a_{L+M+1}$ the generalized Pad\'e approximant reduces to a Pad\'e approximant.
\end{remark}

\begin{remark}
    In particular, relevant here is the case where each point appears twice so that the Pad\'e approximant $\frac{Q_n}{P_n}$ solves the following interpolation problem:
    \begin{equation}\label{eq::InterpolationCondition}
        \begin{aligned}
            &\phi(a_i) = \left(\frac{Q_n}{P_n}\right)(a_i) & \text{ and } && \phi'(a_i) = \left(\frac{Q_n}{P_n}\right)'(a_i)
        \end{aligned}
    \end{equation}
    for each $i =1,\dots,n$.
    In this case, by \cite{Baker}*{Volume II, Theorem 1.1.1} the interpolation condition \eqref{eq::InterpolationCondition} is equivalent to the following
    \begin{equation}
        P_n(x) \phi(x) - Q_n(x) = r(x) \prod_{k=1}^{n}(x-a_k)^2
    \end{equation}
    where the remainder function $r(x)$ is regular at the $a_k$.
    For the general case the remainder function $r(x)$ is listed in \cite{Baker} and can be expressed as a determinant of a matrix the coefficients of which are divided differences of $\phi(x)$ at the points $\{a_i\}$.
    Finally, note that the remainder function and the Pad\'e approximant are uniquely determined by the interpolation condition \eqref{eq::InterpolationCondition} (for further details see 
    \cite{Baker}*{Vol. II, Chp. I }).
\end{remark}

\begin{theorem}
    Let $\phi(x) = 1/\sqrt{x^2 -1}$. Then we have that 
    
    \begin{equation}\label{IPforTU}
        \begin{aligned}
            \phi(a_i) &= \left(\frac{U_{n-1}}{T_n}\right)(a_i), & \text{ and } && \phi'(a_i) =  \left(\frac{U_{n-1}}{T_n}\right)'(a_i)
        \end{aligned}
    \end{equation}
     for each $i=1,\dots,n$. In other words, $U_{n-1}/T_n$ is a multipoint Pad\'e approximant of type $[n-1/n]$ to $\phi$ at $\{a_1,a_1,a_2,a_2,\dots,a_n,a_n\}$.
\end{theorem}

\begin{proof}
   Note 
    \begin{equation}
        \left(\frac{U_{n-1}}{T_n}\right)(x) = \frac{2}{z-z^{-1}}\frac{f_{n}(z){}^{2}-1}{f_n(z){}^{2}+1} =  \phi(x) \frac{1-f_{n}(z){}^{2}}{f_n(z){}^{2}+1} .
    \end{equation}
    The condition on $\phi(a_i)$ follows from the fact $x = a_i$ implies that $z = c_i$.
    The condition on $\phi'(a_i)$ follows from a slightly longer calculation.
    Observe
    \begin{equation}
        \begin{aligned}
            \derivative{}{x} \left(\frac{U_{n-1}}{T_n}\right)(x) &= \phi'(x)\frac{1-f_{n}(z){}^{2}}{f_n(z){}^{2}+1} + \phi(x)\frac{2 z^2}{z^2-1}  \derivative{}{z} \frac{1-f_{n}(z){}^{2}}{f_n(z){}^{2}+1}             
            \\ & = \phi'(x)\frac{1-f_{n}(z){}^{2}}{f_n(z){}^{2}+1} - \phi(x) \frac{2 z^2}{z^2-1} \frac{4 f_n(z) f_n'(z)}{\left(f_n(z){}^2+1\right){}^2}
        \end{aligned}
    \end{equation}
    evaluating at $x=a_i$ gives
    \begin{equation}
        \left(\frac{U_{n-1}}{T_n}\right)'(a_i) = -\frac{a_i}{(a_i^2 -1)^{\frac{3}{2}}}.
    \end{equation} 
\end{proof}
\begin{remark}
Note that the interpolation condition \eqref{IPforTU} is different from the one considered in \cites{DZh,D11}, where each node had multiplicity one while \eqref{IPforTU} means that each node has multiplicity 2.
\end{remark}
\begin{remark}
There is a more general way to see $U_{n-1}/T_n$ is a multipoint Pad\'e approximant. Let $\omega_n(x) = \prod_{k=1}^n (x- a_k)$.
Observe that, 
\begin{equation}
    \begin{aligned}
        \phi(x)T_n(x) &- U_{n-1}(x) =  T_n(x)\int_{-1}^1\frac{1}{t-x} \frac{\dd t }{\pi\sqrt{1-t^2}} -U_{n-1}(x)
        \\ & = \int_{-1}^{1}\frac{T_n(x) - T_n(t)}{t-x} \frac{\dd t }{\pi\sqrt{1-t^2}} + \int_{-1}^{1}\frac{T_n(t)}{t-x} \frac{\dd t }{\pi\sqrt{1-t^2}} -U_{n-1}(x)
       \\ & = \int_{-1}^{1}\frac{T_n(t)}{t-x} \frac{\dd t }{\pi\sqrt{1-t^2}} 
    \end{aligned}
\end{equation}
where the last equality follows from Theorem \ref{prop::OrthogonalityRelationTwo}.
By Proposition \ref{thm::OrthogonalityRelationOne} it can be concluded that $\phi(x)T_n(x) - U_{n-1}(x)$ vanishes at $a_i$ for each $i=1,\dots,n$.
Therefore, the polynomials $P_n(x) = \omega_n(x)T_n$ and $Q_n(x) = \omega_n(x)U_{n-1}$  satisfy
\begin{equation*}
    P_n(x) \phi(x) - Q_n(x) = [\omega_n(x)]^2 r(x).
\end{equation*}
For a function $r(x)$ regular at each of the $a_i$.
Noting that $T_n$ and $U_n$ have zeros only in $[-1,1]$ and therefore which cannot coincide with an interpolation node, it can be concluded that the ratio $\frac{U_{n-1}(x)}{T_n(x)}$ is a generalized Pad\'e approximant of $\phi(x)$. 
Finally, there is another way, which is not as general as the above method, but can be generalized to a larger class of rational functions. 
Namely, \eqref{PellAbel} can be rewritten as
\[
\frac{1}{\sqrt{x^2-1}}-\frac{U_{n-1}(x)}{T_n(x)}=\frac{1}{T_n(x)(T_n(x)+\sqrt{x^2-1}U_{n-1}(x))\sqrt{x^2-1}}
\]
or
\begin{equation}
\label{eq:errorFunction}
    \frac{1}{\sqrt{x^2-1}}-\frac{U_{n-1}(x)}{T_n(x)}= \frac{ [\omega_n(x)]^2}{P_n(x)(P_n(x)+\sqrt{x^2-1}Q_{n}(x))\sqrt{x^2-1}},
\end{equation}
which not only yields the interpolation condition \eqref{IPforTU} but also gives an explicit expression for the error function $r(x)$.
\end{remark}

\begin{theorem}
    Let $\frac{Q_{n-1}(x)}{P_{n}(x)}$ be a type $[n-1/n]$ Pad\'e approximant to a function $\phi(x)$ at the points $\{a_1,a_1,\dots,a_n,a_n\}$.
    Let $\omega_n(x) = \prod_{k=1}^n (x- a_k)$.
    Then,
    \begin{equation*}
    \frac{P_n(x)}{\omega_n(x)} = \begin{vmatrix}
        \phi'(a_1) &  \frac{ \phi(a_1) - \phi(a_2)}{a_1-a_2}  & \frac{ \phi(a_1) - \phi(a_3)}{a_1-a_3} & \dots & \frac{ \phi(a_1) - \phi(a_n)}{a_1-a_n} &\phi(a_1)
    \\    \frac{ \phi(a_2) - \phi(a_1)}{a_2-a_1}  &   \phi'(a_2) &  \frac{ \phi(a_2) - \phi(a_3)}{a_2-a_3} & \dots & \frac{ \phi(a_2) - \phi(a_n)}{a_2-a_n}  &\phi(a_2)
    \\ \vdots &\vdots & \vdots & &\vdots &\vdots
    \\ \frac{ \phi(a_n) - \phi(a_1)}{a_n-a_1}  &  \frac{ \phi(a_n) - \phi(a_2)}{a_n-a_2} &  \frac{ \phi(a_n) - \phi(a_3)}{a_n-a_3} & \dots & \phi'(a_n) &\phi(a_n)
    \\ \frac{1}{x-a_1}& \frac{1}{x-a_2}& \frac{1}{x-a_3} & \dots &\frac{1}{x-a_n}& 1
    \end{vmatrix},
 \end{equation*}

and 

 \begin{equation*} 
    \frac{Q_{n-1}(x)}{\omega_n(x)}  =\begin{vmatrix}
        \phi'(a_1) &  \frac{ \phi(a_1) - \phi(a_2)}{a_1-a_2}  & \frac{ \phi(a_1) - \phi(a_3)}{a_1-a_3} & \dots & \frac{ \phi(a_1) - \phi(a_n)}{a_1-a_n} &\phi(a_1)
    \\    \frac{ \phi(a_2) - \phi(a_1)}{a_2-a_1}  &   \phi'(a_2) &  \frac{ \phi(a_2) - \phi(a_3)}{a_2-a_3} & \dots & \frac{ \phi(a_2) - \phi(a_n)}{a_2-a_n}  &\phi(a_2)
    \\ \vdots &\vdots & \vdots & &\vdots &\vdots
    \\ \frac{ \phi(a_n) - \phi(a_1)}{a_n-a_1}  &  \frac{ \phi(a_n) - \phi(a_2)}{a_n-a_2} &  \frac{ \phi(a_n) - \phi(a_3)}{a_n-a_3} & \dots & \phi'(a_n) &\phi(a_n)
    \\ \frac{\phi(a_1)}{x-a_1}& \frac{\phi(a_2)}{x-a_2}& \frac{\phi(a_3)}{x-a_3} & \dots &\frac{\phi(a_n)}{x-a_n}& 0
    \end{vmatrix}.
    \end{equation*}
\end{theorem}
\begin{proof}
    The rational functions $\frac{P_n(x)}{\omega_n(x)}$ and $\frac{Q_{n-1}(x)}{\omega_n(x)}$ can be expanded as partial fractions as
    \begin{equation}
       \begin{aligned}
         \frac{P_n(x)}{\omega_n(x)} &= A_{0,n} + \sum_{k=1}^{n} \frac{A_{k,n}}{x-a_{k}},
      \\   \frac{Q_n(x)}{\omega_n(x)} &= \sum_{k=1}^{n} \frac{B_{k,n}}{x-a_{k}}
       \end{aligned}
    \end{equation}
    for some $A_{0,n},A_{1,n},\dots A_{n,n}$. 
    Define,
    \begin{equation}\label{eq::MkandWk}
    \begin{aligned}
        w_{k}(x)  &= \sum_{i\neq k} \frac{B_{i,n}}{x-a_i},
     \\ m_{k}(x)  &= A_{0,n} + \sum_{i\neq k} \frac{A_{i,n}}{x-a_i}.
    \end{aligned}
\end{equation}
The first set of interpolation conditions, which constrain the values of the Pad\'e approximant at the interpolation nodes, provides a simple constraint on the ratio of $A_{i,n}$ and $B_{i,n}$.
Observe
    \begin{equation}\label{eq::FirstInterpolationCondition}
       \phi(a_i) =  \lim_{x\to a_i} \frac{Q_{n-1}(x)}{P_n(x)} = \lim_{x\to a_i}  \frac{ \frac{B_{i,n}}{x-a_i} + w_{i}(x) }{\frac{A_{i,n}}{x-a_i} + m_{i}(x) } = \frac{B_{i,n}}{A_{i,n}}
    \end{equation}
    follows from the fact that $m_i,w_i$ are regular at $a_i$. 
The second interpolation condition, which constrains the derivative of the Pad\'e approximant at the interpolation nodes, provides a more complicated set of equations for $A_{i,n},B_{i,n}$.
They are
\begin{equation}
    \phi'(a_i) =\lim_{x\to a_i} \derivative{}{x} \frac{Q_{n-1}(x)}{P_{n}(x)} = \frac{A_{i,n} w_{i}(a_i) -B_{i,n} m_i(a_i)}{A_{i,n}^2}
\end{equation}
for $i=1,\dots,n$.
Using \eqref{eq::FirstInterpolationCondition} as well as the definitions of $m_i,w_i$ this becomes
\begin{equation}\label{eq::SecondInterpolationCondition}
    \begin{aligned}
        \phi'(a_k)A_{k,n} &= w_{k}(a_k) - \phi(a_k) m_k(a_k)
        \\ & =  \sum_{i\neq k} \frac{\phi(a_i) A_{i,n}}{a_k-a_i} - \sum_{i\neq k} \frac{\phi(a_k) A_{i,n}}{a_k-a_i} - A_{0,n} \phi(a_k)
        \\ & =  -\sum_{i\neq k} \frac{ \phi(a_k) - \phi(a_i)}{a_k-a_i} A_{i,n} - A_{0,n} \phi(a_k)
    \end{aligned}
\end{equation}
 Observe that \eqref{eq::SecondInterpolationCondition} describes the linear system of equations:
    \begin{equation*}
    \begin{pmatrix}
        \phi'(a_1) &  \frac{ \phi(a_1) - \phi(a_2)}{a_1-a_2}  & \frac{ \phi(a_1) - \phi(a_3)}{a_1-a_3} & \dots & \frac{ \phi(a_1) - \phi(a_n)}{a_1-a_n} 
    \\    \frac{ \phi(a_2) - \phi(a_1)}{a_2-a_1}  &   \phi'(a_2) &  \frac{ \phi(a_2) - \phi(a_3)}{a_2-a_3} & \dots & \frac{ \phi(a_2) - \phi(a_n)}{a_2-a_n} 
    \\ \vdots &\vdots & \vdots & &\vdots
    \\ \frac{ \phi(a_n) - \phi(a_1)}{a_n-a_1}  &  \frac{ \phi(a_n) - \phi(a_2)}{a_n-a_2} &  \frac{ \phi(a_n) - \phi(a_3)}{a_n-a_3} & \dots & \phi'(a_n) 
    \end{pmatrix}
    \begin{pmatrix}
        A_{1,n}
     \\ A_{2,n}
     \\\vdots
     \\ A_{n,n}   
    \end{pmatrix} = 
    A_{0,n} \begin{pmatrix}
        \phi(a_1)
     \\ \phi(a_2)
     \\\vdots
     \\ \phi(a_n)   
    \end{pmatrix}.
\end{equation*}
Let 
\begin{equation}
   M =  \begin{pmatrix}
        \phi'(a_1) &  \frac{ \phi(a_1) - \phi(a_2)}{a_1-a_2}  & \frac{ \phi(a_1) - \phi(a_3)}{a_1-a_3} & \dots & \frac{ \phi(a_1) - \phi(a_n)}{a_1-a_n} 
    \\    \frac{ \phi(a_2) - \phi(a_1)}{a_2-a_1}  &   \phi'(a_2) &  \frac{ \phi(a_2) - \phi(a_3)}{a_2-a_3} & \dots & \frac{ \phi(a_2) - \phi(a_n)}{a_2-a_n} 
    \\ \vdots &\vdots & \vdots & &\vdots
    \\ \frac{ \phi(a_n) - \phi(a_1)}{a_n-a_1}  &  \frac{ \phi(a_n) - \phi(a_2)}{a_n-a_2} &  \frac{ \phi(a_n) - \phi(a_3)}{a_n-a_3} & \dots & \phi'(a_n) 
    \end{pmatrix}
\end{equation}
and let $M_{i}$ be $M$ with the $i$th column replaced with $ (\phi(a_1) \,\,\, \dots \,\,\, \phi(a_n))^T$
then by Cramer's rule
\begin{equation}
    A_{i,n} = A_{0,n} \frac{\det M_i }{ \det M}.
\end{equation}
Now we have that,
\begin{equation}\label{eq::CofactorExpansion}
    T_n(x) = \frac{ A_{0,n} }{ \det M} \left( \det M + \frac{\det M_1}{x-a_1} +\dots +  \frac{\det M_n }{x-a_n} \right)
\end{equation}
By picking $A_{0,n} =  \det M$ and observing that equation \eqref{eq::CofactorExpansion} is the cofactor expansion of the determinant in the statement of the theorem proves the theorem for $P_n(x)$.
The result for $Q_{n-1}(x)$ follows from a similar observation as well as the interpolation condition \eqref{eq::FirstInterpolationCondition}.
\end{proof}

\section{Convergence of multipoint Pad\'e approximants}
\label{sec::Convergence}

In this section we will demonstrate that the ratio $\dfrac{U_{n-1}}{T_{n}}$ converges in some cases. 
\begin{theorem}\label{thm:ConvergenceRealNodes}
Let $a_k\in{\mathbb{R}}\setminus[-1,1]$ and let the corresponding $c_k$ satisfy
\[
\sum_{k=1}^{\infty}(1-|c_k|)=\infty.
\]
Then the multipoint Pad\'e approximant $\dfrac{U_{n-1}}{T_{n}}$ converges to $\phi$ locally uniformly on ${\mathbb C}\setminus[-1,1]$.
\end{theorem}
\begin{proof}
Note that
\[
\frac{U_{n-1}(x)}{T_n(x)}=\frac{2}{z-z^{-1}}\frac{f^2_{n}(z)-1}{f^2_n(z)+1}=\frac{2}{z-z^{-1}}\left(1-\frac{2}{f^2_n(z)+1}\right).
\]
Using the standard fact about Blaschke products we conclude that 
\[
\frac{2}{f^2_n(z)+1}
\]
converges to 0 locally uniformly on ${\mathbb C}\setminus[-1,1]$, which completes the proof by taking into account that 
\[
\phi(x)=\frac{2}{z-z^{-1}}.
\]
\end{proof}
We can give a continued fraction version of the above theorem.

\begin{corollary} Let $a_k\in{\mathbb{R}}\setminus[-1,1]$ and let the corresponding $c_k$ satisfy
\[
\sum_{k=1}^{\infty}(1-|c_k|)=\infty.
\]
Then
\begin{equation}\phi(x)=\frac{1}
{\frac{a_1x-1}{\sqrt{a_1^2-1}}-
 \displaystyle{\frac{(x-a_1)^2/(a_1^2-1)} {\frac{a_1x-1}{\sqrt{a_1^2-1}}+\frac{a_2x-1}{\sqrt{a_2^2-1}}-
\displaystyle{\frac{(x-a_2)^2/(a_2^2-1)}{{\frac{a_2x-1}{\sqrt{a_2^2-1}}+\frac{a_3x-1}{\sqrt{a_3^2-1}}-
\displaystyle{\frac{(x-a_3)^2/(a_3^2-1)}{\ddots
}}}}}}}},
\end{equation}
that is the continued fraction converges and it converges to $\phi(x)$ locally uniformly on ${\mathbb C}\setminus[-1,1]$.
\end{corollary}

We can take the convergence beyond the Newtonian interpolating scheme and prove it for a triangular sequence of nodes. However, before we can proceed we need to make one observation.

\begin{remark}
Note that having $a_k=1$ or $a_k=-1$ for some $k$ does not create an issue for the existence of $T_n$ or $U_n$ as it simply means that the corresponding Blaschke product in $f_n$ becomes 
$\pm 1$. Hence, it simply reduces the order of the polynomial and the number of interpolating conditions that the ratio has. As result, when a sequence of interpolating nodes has $\pm 1$ we just disregard them as they do not produce interpolating conditions. 
\end{remark}

\begin{theorem}
Let $a_{n,k}$ be the n-th roots of unity.
Then the multipoint Pad\'e approximant $\dfrac{U_{n-1}}{T_{n}}$ converges to $\phi$ locally uniformly on ${\mathbb C}\setminus[-1,1]$.
\end{theorem}
\begin{proof}
As before, we have that
\[
\frac{U_{n-1}(x)}{T_n(x)}=\frac{2}{z-z^{-1}}\left(1-\frac{2}{f^2_n(z)+1}\right).
\]
Note that, $f_n$ is in fact a finite Blaschke product since the corresponding set $\{c_{n,k}\}$ contains complex conjugate pairs in the product and therefore can be rearranged into a finite Blaschke product:
\[
f_n(w)=\prod_{k=1}^{n} b_{n,k}(w), \quad b_{n,k}(w)=\frac{w-c_{n,k}}{1-\overline{c}_{n,k}w},
\]
where $w=1/z\in{\mathbb D}$. 
Next, we have
\[
|f_n(w)|\le\exp\left(-\sum_{k=1}^{n}(1-|b_{n,k}(w)|)\right).
\]
At the same time, we have that
\[
1-|b_{n,k}(w)|\ge C(1-|c_{n,k}|)
\]
on a closed subset of ${\mathbb D}$. Since the roots of unity fill in the unit circle in the limit, the numbers $c_{n,k}$ fill in the corresponding part of the curve $|z+1/z|=2$. As a result, we have
\[
\lim_{n\to\infty}\sum_{k}^{n}(1-|c_{n,k}|)=\infty,
\]
which yields $f_n\to 0$. Then returning to $z$, we get that $f_n(z)\to\infty$ when $|z|>1$ and is contained in a compact set.
\end{proof}

\section{Convergence in the presence of noise.}
\label{sec::EffectOfNoise}

The presence of noise in the interpolation data affects the convergence properties of Pad\'e and multi-point Pad\'e approximants.
We have performed the following numerical experiments, inspired by those in \cite{Costin:2022hgc} for single-point Pad\'e approximants, and we observe qualitatively similar results. 
Random noise, with a controlled strength, is added to the input interpolation data to analyze the effect on the multi-point Pad\'e approximants for the function $\phi(x)=1/\sqrt{x^2 - 1}$ in \eqref{eq::phi}.

Let $r_i,r_i' \in \C $ be complex random variables uniformly distributed in a square centered at the origin with side length 2.
Using Kronecker's algorithm \cite{Baker} implemented in Mathematica, the $[n-1/n]$ approximants were calculated for various configurations of interpolation nodes $\{a_1,a_1,\dots,a_n,a_n\}\in \C\setminus[-1,1]$ for the function $\phi_{\epsilon}(x)$ defined by the input  interpolation data:
\begin{equation}\label{eq::NoisyInterpolationProblem}
    \begin{aligned}
        \phi_\epsilon(a_i) &= \phi(a_i) + \epsilon\, r_i ,
    \\    \phi_\epsilon'(a_i) &= \phi'(a_i) + \epsilon\, r_i' .
    \end{aligned}
\end{equation}
The level of noise can be controlled by varying the small parameter $\epsilon$.

The breakdown of multipoint Pad\'e approximants found here is analogous to the breakdown of Pad\'e approximants to series which was analyzed in \cite{Costin:2022hgc}. We find that at a certain threshold strength of the noise the multipoint Pad\'e approximants break down due to the appearance of spurious pairs of poles and zeros (or Froissart doublets).  
For a fixed set of interpolation nodes and for $\epsilon$ sufficiently small, the approximants have poles and zeros indistinguishable from the unperturbed multi-point Pad\'e approximant.
As $\epsilon$ increases,  there is a threshold multipoint Pad\'e order $n_c$ at which spurious poles and zeros appear. Furthermore, numerical experiments indicate that the breakdown order is proportional to $\log_{10} \epsilon$. This is consistent with \cite{Costin:2022hgc} for single-point Pad\'e for series extrapolation, where the logarithmic behavior was predicted using the asymptotic properties of the approximate conformal map that single-point Pad\'e effectively generates.
 See Figure \ref{fig::Breakdown}, which includes estimates of the breakdown for two triangular interpolation schemes.
The first set is given by points on the real line, which lie in the interval $[2,4]$, and are given by $a_i \in \left\{2+2m/(n-1)\right\}_{i=0}^{n-1}$ for each $n$. The second set is given by the $n$ roots of $(2x/3)^n = 1$ which lie on the circle of radius $3/2$.
\begin{figure}
    \includegraphics[width=\textwidth]{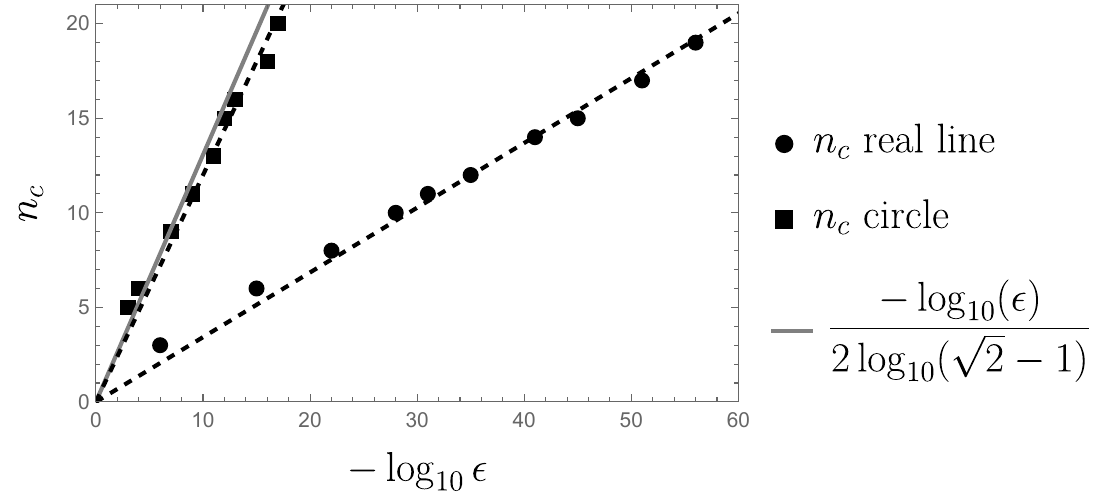}
    \caption{
        This plot shows the breakdown Pad\'e order $n_c$ for the function $\phi(x)=1/\sqrt{x^2 -1}$, as a function of the noise strength $\epsilon$, for two different configurations of interpolation nodes. 
        The black squares are for nodes distributed evenly around the circle with radius $3/2$, and the black circles are for nodes distributed evenly on the interval $[2,4]$ of the real line. 
        Dashed lines show fits to these two sets of threshold points, indicating logarithmic growth as in \cite{Costin:2022hgc}. 
        The solid line shows the result for single-point Pad\'e predicted by \cite{Costin:2022hgc} for the series case of the Pad\'e approximant, in which series data was provided at a single point. 
        Notice that the slope for the circular configuration of interpolation nodes matches very closely to the slope for single-point Pad\'e predicted by \cite{Costin:2022hgc}.}
    \label{fig::Breakdown}
\end{figure}

%
%

An important new feature of the multi-point Pad\'e interpolation noise threshold is that the slope of the logarithmic behavior depends on the distribution of the input interpolation nodes.  This can be seen in  Figure \ref{fig::Breakdown}, and the effect of different nodal distributions is also illustrated for various other configurations of input nodes in Figures  \ref{fig::RealLineBreakdown}, \ref{fig::ImAxisBreakdown}, \ref{fig::SymImAxisBreakdown} and \ref{fig::CircleBreakdown}. Figure  \ref{fig::RealLineBreakdown} shows the case where the interpolation nodes (hollow circles) are placed on the real line on one side of the natural cut $[-1, 1]$, where in the presence of very weak noise the Pad\'e poles (solid circles) and zeros (hollow triangles) accumulate. But as the level of noise increases the Pad\'e poles and zeros migrate to the vicinity of the interpolation points.
Figure  \ref{fig::ImAxisBreakdown} shows the case where the interpolation nodes (hollow circles) are placed on the imaginary axis on one side of the natural cut $[-1, 1]$, where in the presence of very weak noise the Pad\'e poles and zeros accumulate. Note that the Pad\'e poles form a cut that is curved outwards on the other side of the real line. Once again, as the level of noise increases the Pad\'e poles and zeros migrate to the vicinity of the interpolation points on the imaginary axis. In Figure \ref{fig::SymImAxisBreakdown} the interpolation nodes are placed on the imaginary axis symmetrically about the real axis. With weak noise the Pad\'e poles and zeros accumulate on a straight line cut $[-1, 1]$, in contrast to the asymmetric curved cut in Figure  \ref{fig::ImAxisBreakdown}, and as the noise level increases the Pad\'e poles and zeros accumulate to regions near the two distinct regions of interpolation nodes on the positive and negative imaginary axis. In Figure \ref{fig::CircleBreakdown} the input nodes are placed symmetrically around a circle centered on the origin and with radius $3/2$. This circular  configuration of input interpolation nodes is less susceptible to noise than those in Figures  \ref{fig::RealLineBreakdown}, \ref{fig::ImAxisBreakdown} and \ref{fig::SymImAxisBreakdown}.

 The behaviors illustrated in Figures  \ref{fig::RealLineBreakdown}, \ref{fig::ImAxisBreakdown}, \ref{fig::SymImAxisBreakdown} and \ref{fig::CircleBreakdown} are representative of a generic feature of the Pad\'e  pole and zero distributions when subject to noise: the spurious poles and zeros eventually accumulate in the vicinity of the interpolation nodes. This should be contrasted with the Pad\'e analysis of series, where the spurious poles and zeros accumulate to circular natural boundaries associated with the actual singularities \cite{Costin:2022hgc}.

\begin{figure}
    \begin{subfigure}[a]{0.45\textwidth}
        \includegraphics[width=\textwidth]{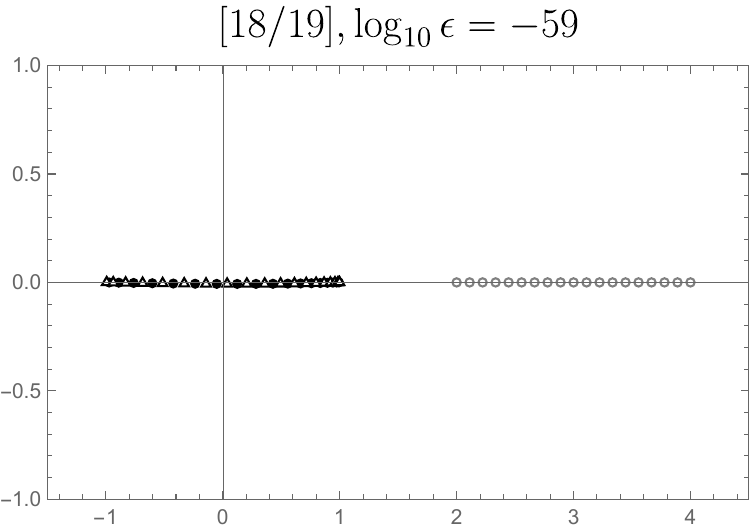}    
        \caption{}
    \end{subfigure}
    \begin{subfigure}[a]{0.45\textwidth}
        \includegraphics[width=\textwidth]{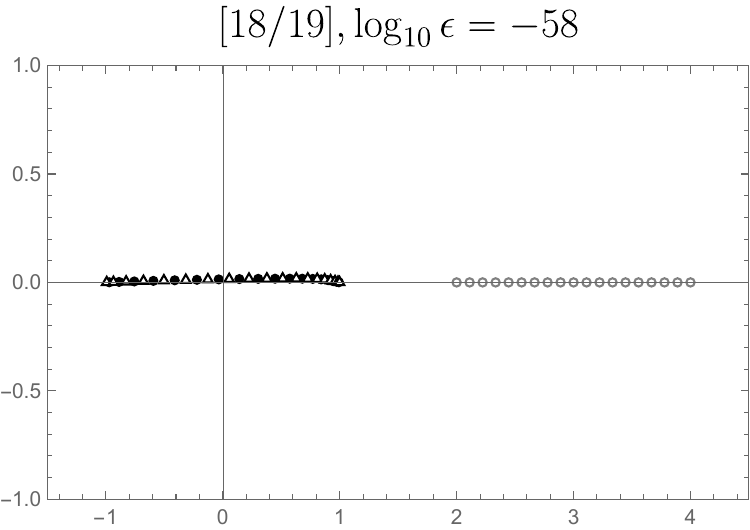}   
        \caption{} 
    \end{subfigure}
    \begin{subfigure}[a]{0.45\textwidth}
        \includegraphics[width=\textwidth]{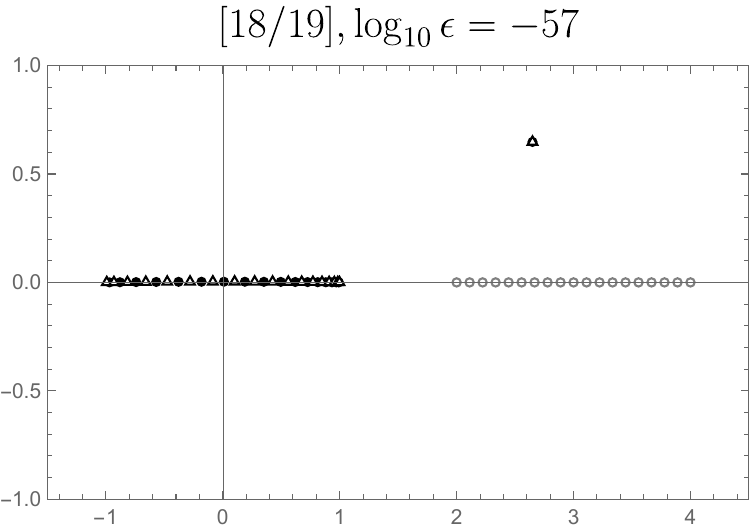}    
        \caption{}
    \end{subfigure}
    \begin{subfigure}[a]{0.45\textwidth}
        \includegraphics[width=\textwidth]{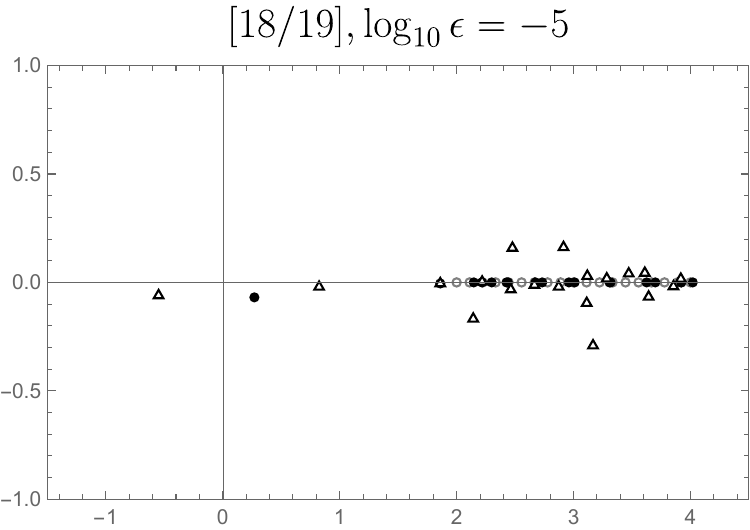}    
        \caption{}
    \end{subfigure}
   
    \caption{ The plots show the zeros (open triangles $\mathbf{\Delta}$) and poles (solid circles $\bullet$) of the type $[18/19]$ multipoint Pad\'e approximant for the interpolation problem \eqref{eq::NoisyInterpolationProblem}, with the input interpolation nodes $a_i$ placed on the positive real axis, as indicated by the open circles {\color{gray} $\circ$}. The first spurious poles appear with noise level $\epsilon\approx 10^{-57}$, and note that as the noise level increases the Pad\'e poles and zeros migrate to the vicinity of the input interpolation nodes. }
    \label{fig::RealLineBreakdown}
\end{figure}

\begin{figure} 
    \begin{subfigure}[a]{0.45\textwidth}
        \includegraphics[width=\textwidth]{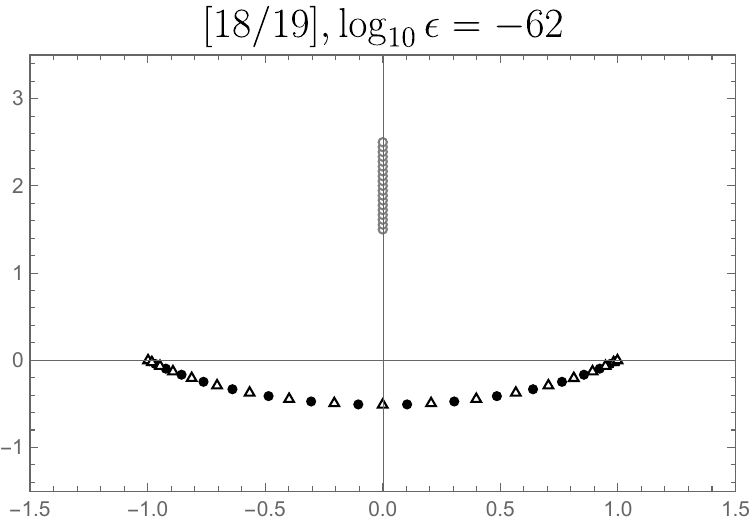}    
        \caption{}
    \end{subfigure}
    \begin{subfigure}[a]{0.45\textwidth}
        \includegraphics[width=\textwidth]{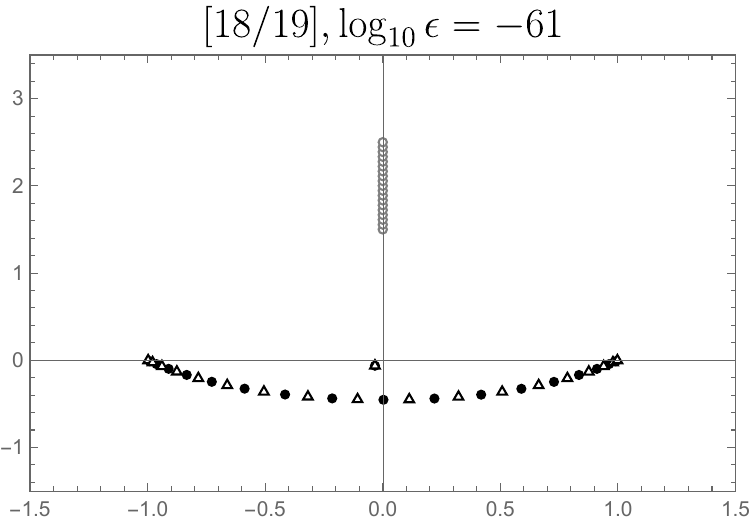}   
        \caption{} 
    \end{subfigure}
    \begin{subfigure}[a]{0.45\textwidth}
        \includegraphics[width=\textwidth]{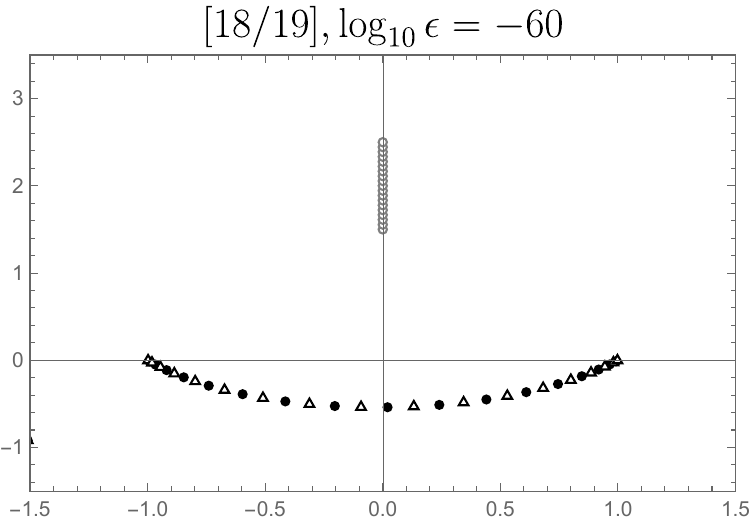}    
        \caption{}
    \end{subfigure}
    \begin{subfigure}[a]{0.45\textwidth}
        \includegraphics[width=\textwidth]{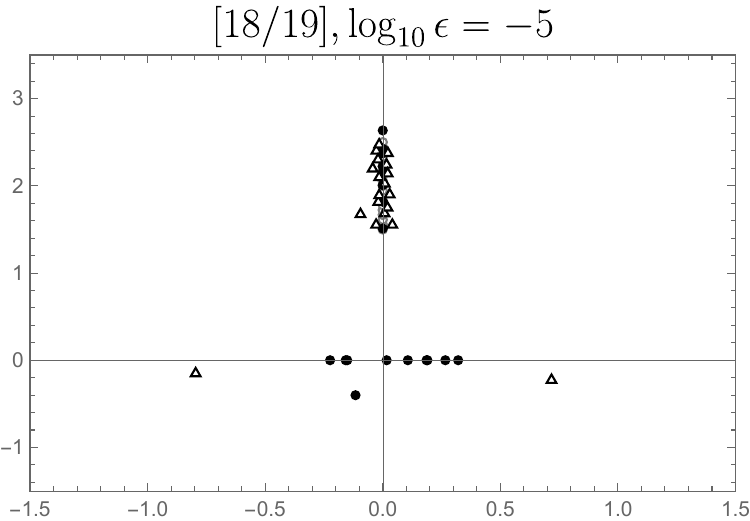}    
        \caption{}
    \end{subfigure}
  
    \caption{
    The plots show the zeros (open triangles $\mathbf{\Delta}$) and poles (solid circles $\bullet$) of the type $[18/19]$ multipoint Pad\'e approximant for the interpolation problem \eqref{eq::NoisyInterpolationProblem}, with the input interpolation nodes $a_i$ placed on the positive imaginary axis, as  indicated by the open circles {\color{gray} $\circ$}. The first spurious poles appear with noise level $\epsilon\approx 10^{-61}$, and note that as the noise level increases the Pad\'e poles and zeros migrate to the vicinity of the input interpolation nodes. The curved arc of low-noise Pad\'e poles and zeros reflects the asymmetry of the input nodes relative to the natural cut.}
    \label{fig::ImAxisBreakdown}
\end{figure}

\begin{figure}
    \begin{subfigure}[a]{0.45\textwidth}
        \includegraphics[width=\textwidth]{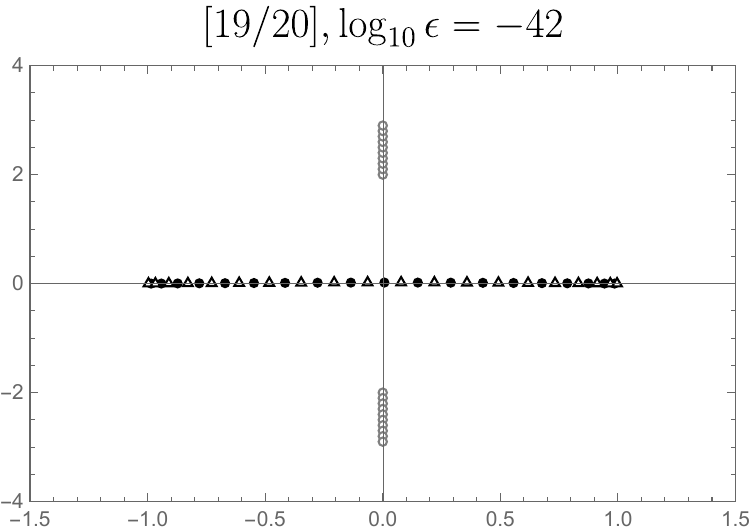}    
        \caption{}
    \end{subfigure}
    \begin{subfigure}[a]{0.45\textwidth}
        \includegraphics[width=\textwidth]{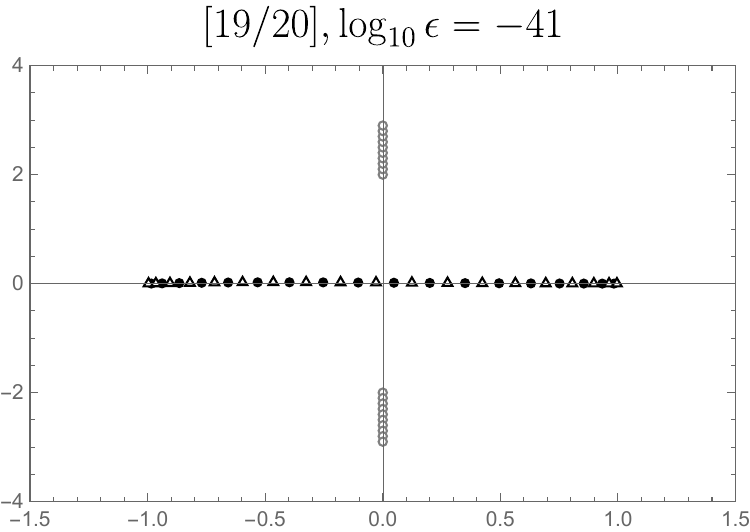}   
        \caption{} 
    \end{subfigure}
    \begin{subfigure}[a]{0.45\textwidth}
        \includegraphics[width=\textwidth]{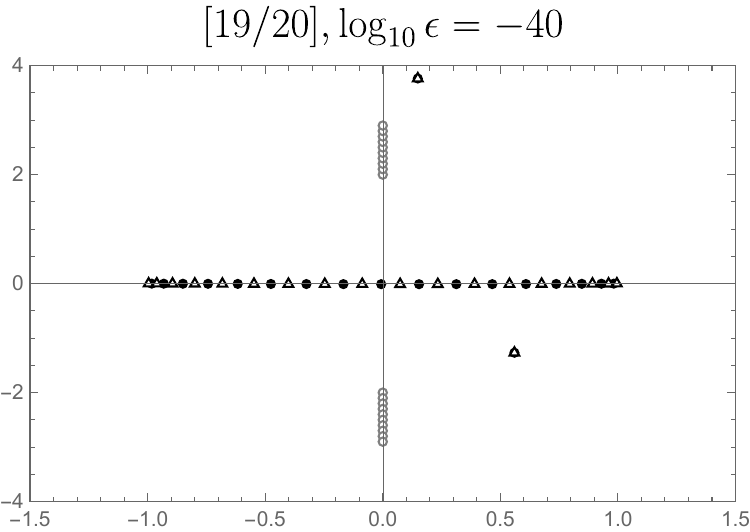}    
        \caption{}
    \end{subfigure}
    \begin{subfigure}[a]{0.45\textwidth}
        \includegraphics[width=\textwidth]{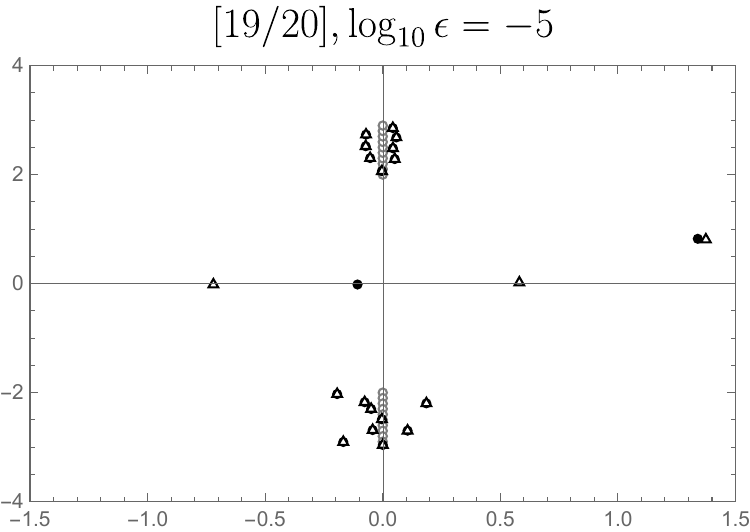}    
        \caption{}
    \end{subfigure}
   
    \caption{ 
 The plots show the zeros (open triangles $\mathbf{\Delta}$) and poles (solid circles $\bullet$) of the type $[19/20]$ multipoint Pad\'e approximant for the interpolation problem \eqref{eq::NoisyInterpolationProblem}  with the input interpolation nodes $a_i$ placed symmetrically on the positive and negative imaginary axis, as  indicated by the open circles {\color{gray} $\circ$}. 
 The first spurious poles appear with noise level $\epsilon\approx 10^{-40}$, and note that as the noise level increases the Pad\'e poles and zeros migrate to the vicinity of the input interpolation nodes. 
 The low-noise Pad\'e poles and zeros form a straight line, rather than curved as in Figure \ref{fig::ImAxisBreakdown}, reflecting the symmetry of the input nodes relative to the natural cut.
    \label{fig::SymImAxisBreakdown}}
\end{figure}

\begin{figure}
    \begin{subfigure}[a]{0.45\textwidth}
        \includegraphics[width=\textwidth]{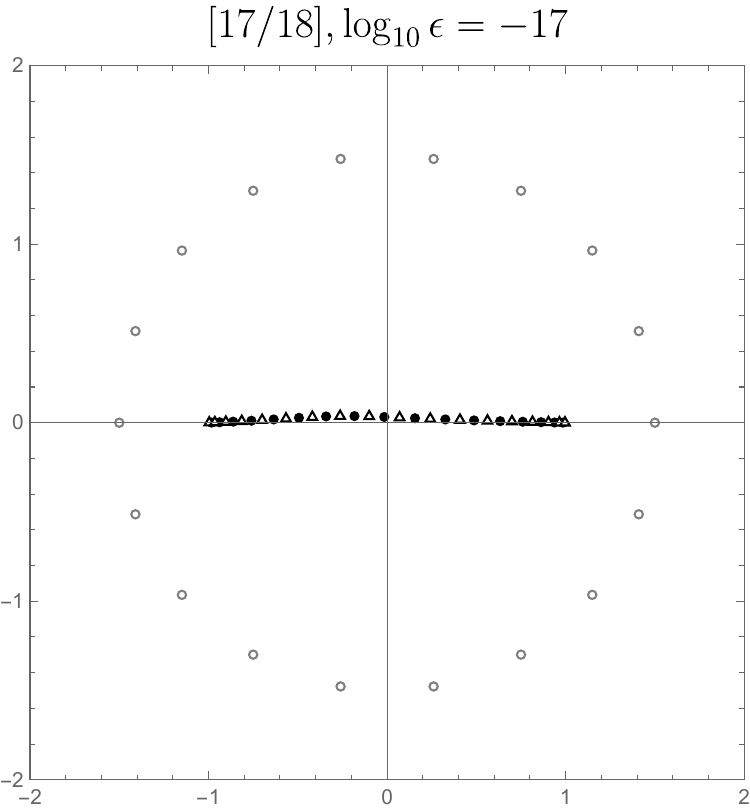}    
        \caption{}
    \end{subfigure}
    \begin{subfigure}[a]{0.45\textwidth}
        \includegraphics[width=\textwidth]{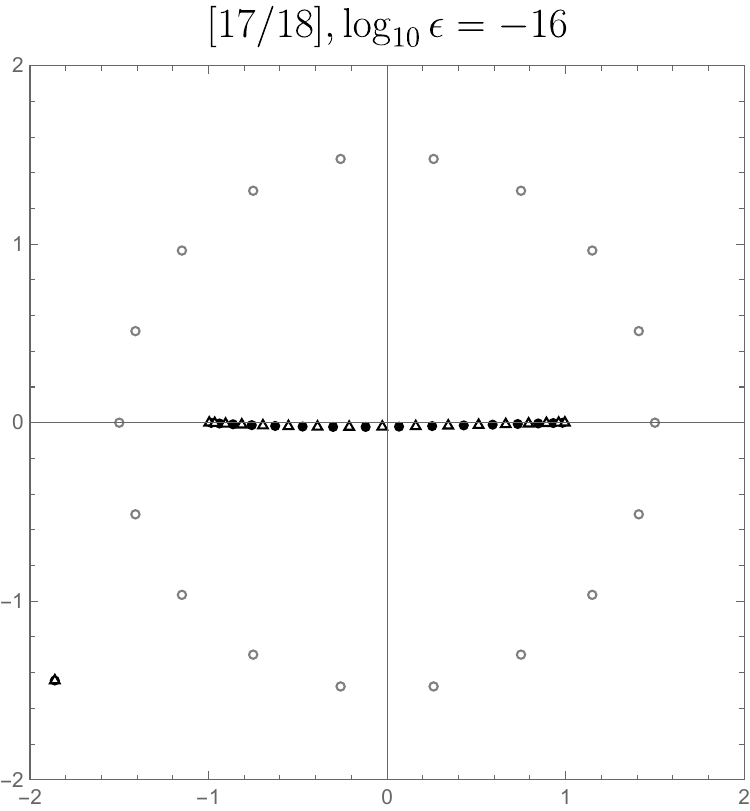}   
        \caption{} 
    \end{subfigure}
    \begin{subfigure}[a]{0.45\textwidth}
        \includegraphics[width=\textwidth]{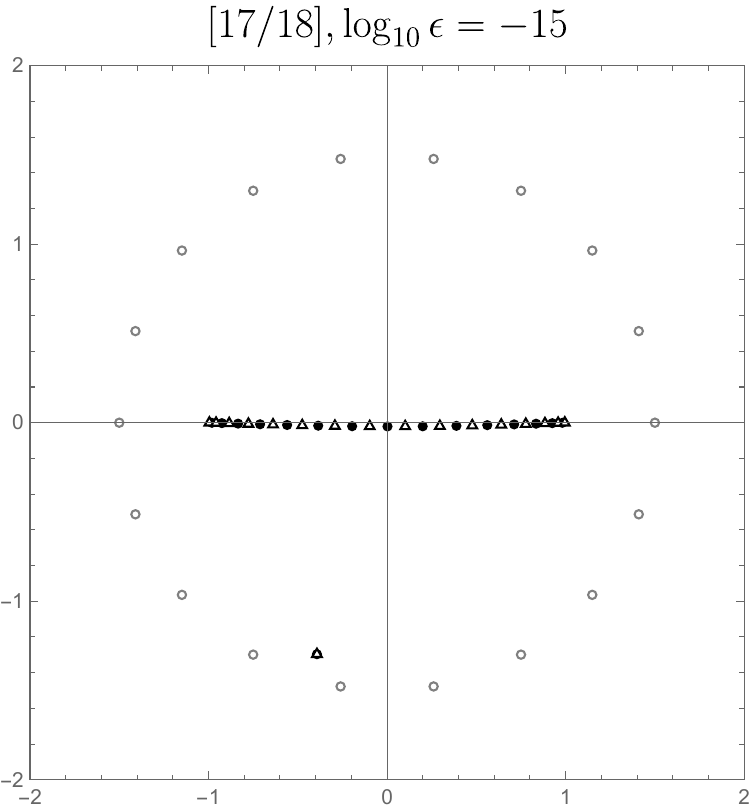}    
        \caption{}
    \end{subfigure}
    \begin{subfigure}[a]{0.45\textwidth}
        \includegraphics[width=\textwidth]{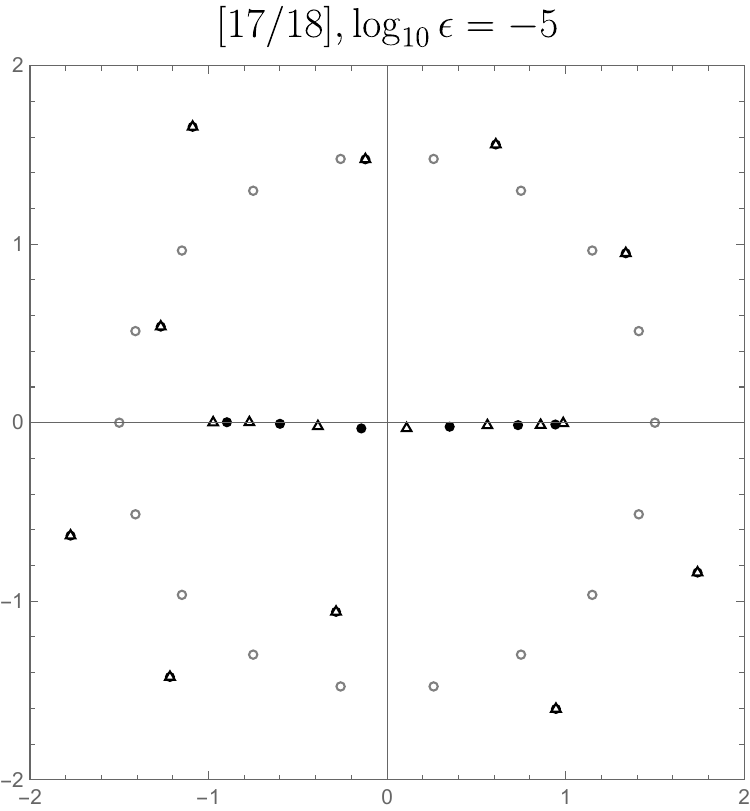}    
        \caption{}
    \end{subfigure}
   
    \caption{
 The plots show the zeros (open triangles $\mathbf{\Delta}$) and poles (solid circles $\bullet$) of the type $[19/20]$ multipoint Pad\'e approximant for the interpolation problem \eqref{eq::NoisyInterpolationProblem} with the input interpolation nodes $a_i$ placed symmetrically around a circle of radius $3/2$, as indicated by the open circles {\color{gray} $\circ$}. 
 The first spurious poles appear with noise level $\epsilon\approx 10^{-15}$, and note that as the noise level increases the Pad\'e poles and zeros migrate to the vicinity of the input interpolation nodes, with some remaining near the cut. The low-noise Pad\'e poles and zeros form a straight line, reflecting the symmetry of the input nodes relative to the natural cut. 
    \label{fig::CircleBreakdown}
    }
\end{figure}

\section{Application to analytic continuation in physics}

The applicaiton of Pad\'e aprpoximants to lattice QCD has received recent attention, see for example \cites{Pasztor_2021,Dimopoulos_2022,PadeGminusTwo,PhysRevD.108.074516,clarke2024searchingqcdcriticalendpoint,Schmidt_2023,Gokce,Goswami:2024jlc}.
A primary reason for the interest in multipoint Pad\'e approximants coming from this community is due to the well documented ability of Pad\'e approximants to analytically continue functions when given only partial information such as a number of series coefficients or interpolation conditions.
As discussed in \cite{Schmidt_2023}, the `sign-problem' (for a review which discusses the `sign problem' see \cite{Nagata_2022}) makes it impossible to perform lattice calculations at real (and physically relevant) values. 
Instead, calculations are done (at some finite number of) points lying on the imaginary axis.
The extraction of physically relevant content---such as the locations of the singularities lying on the real axis---requires analytic continuation.

In this section we provide a rigourous investigation and discussion of the such analytic continuation properties of the function $\phi(x) = \frac{1}{\sqrt{x^2-1}}$ whose multipoint Pad\'e approximants are studied in this article.
To complement the discussion of noise, especially the experiment illustrated in Figure \ref{fig::ImAxisBreakdown}, (which is similar to the scenario relevant to lattice calculations such as described in \cite{Schmidt_2023}), a discussion of the error along the real axis given interpolation nodes placed uniformly in an interval along the imaginary axis is provided.

Firstly, we remark that multipoint Pad\'e can possess advantages over Pad\'e approximants of series data.
Along the real line the Pad\'e approximant to $\phi(x)$ at $x=0$ converges only on the interval $(-1,1)$.
However, the multipoint Pad\'e approximant found from interpolation conditions on the imaginary axis converges on $\mathbb{R}\setminus\{-1,1\} $.

This is illustrated in figure \ref{fig::ConvergenceOnRealLine} where the relative error, given by
\begin{equation}\label{eq::RelativeErrorFunction}
    \epsilon_n(x) =  |\phi(x) - \frac{U_{n-1}(x)}{T_{n}(x)} |/|\phi(x)|,
\end{equation}  
is shown for interpolation schemes of the form
\begin{equation}\label{eq:InterpolationScheme}
    \begin{aligned}
        a_{m+1} &= i \frac{\pi}{n} m , && m = 0,2,\dots,n
    \end{aligned}
\end{equation}
which are similar to those discussed in lattice QCD applications (for example \cite{clarke2024searchingqcdcriticalendpoint}).

\begin{figure}
    \begin{subfigure}[a]{\textwidth}
        \includegraphics[width=\textwidth]{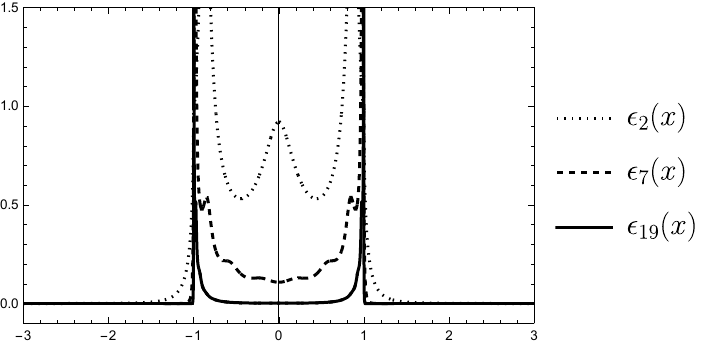}    
        \caption{\label{fig::a_ConvergenceOnRealLine}}
    \end{subfigure}
    \begin{subfigure}[a]{1\textwidth}
        \includegraphics[width=\textwidth]{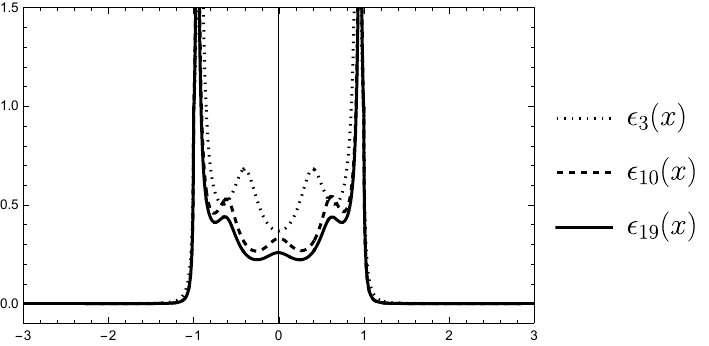}   
        \caption{\label{fig::b_ConvergenceOnRealLine}} 
    \end{subfigure}
    \caption{
    Figure \ref{fig::a_ConvergenceOnRealLine} shows the relative error function \ref{eq::RelativeErrorFunction} for three different numbers $n$ of interpolation nodes described by the interpolation scheme \ref{eq:InterpolationScheme}.
    Figure \ref{fig::b_ConvergenceOnRealLine} is identical to \ref{fig::a_ConvergenceOnRealLine} however, noise with strength $-\log_{10}\epsilon = 15$ was added to the Pad\'e approximant in an identical manner to what was done in section \ref{sec::EffectOfNoise}.
    The plots show fast convergence on $(-\infty ,-1)\cup (1,\infty)$ both with and without noise and slower convergence within the interval $(-1,1)$ without noise.
    With noise the convergence of the Pad\'e approximant halts, including additional interpolation conditions provides little gains in accuracy.
    \label{fig::ConvergenceOnRealLine}
    }
\end{figure}

\vspace{0.5cm}
\noindent {\bf Acknowledgments.} The authors thank Gerald Dunne for fruitful discussions and suggestions that improved the presentation of the manuscript. M.D.'s part of research was partially supported by the NSF DMS grant 2008844. 
M.M.'s part of the research was supported in part by the U.S. Department of Energy, Office of High Energy Physics, Award DE-SC0010339.

\bibliography{/Users/max/Library/CloudStorage/OneDrive-UniversityofConnecticut/ChebyshevFunctions/ChebyshevDraft/refs.bib}

\end{document}